\documentclass[11pt]{article}
\usepackage[top=1in, bottom=1in, left=1in, right=1in]{geometry}

\usepackage[linktocpage,colorlinks,linkcolor=blue,anchorcolor=blue,citecolor=blue,urlcolor=blue,pagebackref]{hyperref}
\usepackage{euscript,mdframed}
\usepackage{microtype,todonotes,relsize}
\usepackage{amsmath,amsthm}
\usepackage{algpseudocode}

\usepackage{accents}
\usepackage{comment,url,graphicx,relsize}
\usepackage{amssymb,amsfonts,amsmath,amsthm,amscd,dsfont,mathrsfs,mathtools,nicefrac, bm}
\usepackage{float,psfrag,epsfig,color,xcolor,url}
\usepackage{epstopdf,bbm,mathtools,enumitem}
\usepackage{subfigure}
\usepackage[ruled,vlined]{algorithm2e}
\usepackage{tablefootnote}
\graphicspath{{Plots/}}
\usepackage{diagbox}
\usepackage{tikz}
\usetikzlibrary{calc,patterns,angles,quotes}
\usepackage{makecell}
\usepackage{multirow}
\usepackage{booktabs}

\newcommand{\E}{\mathbb{E}}
\newcommand{\RR}{\mathbb{R}}

\providecommand{\argmax}{\mathop\mathrm{arg max}} 
\providecommand{\argmin}{\mathop\mathrm{arg min}}

\providecommand{\tr}{\mathop\mathrm{tr}}

\ifdefined\nonewproofenvironments\else
\ifdefined\ispres\else
\renewenvironment{proof}{\noindent\textbf{Proof.}\hspace*{.3em}}{\qed\\}
\newenvironment{proof-sketch}{\noindent\textbf{Proof Sketch}
  \hspace*{0.em}}{\qed\bigskip\\}
\newenvironment{proof-idea}{\noindent\textbf{Proof Idea}
  \hspace*{0.em}}{\qed\bigskip\\}
\newenvironment{proof-of-lemma}[1][{}]{\noindent\textbf{Proof of Lemma {#1}.}
  \hspace*{0.em}}{\qed\\}
\newenvironment{proof-of-corollary}[1][{}]{\noindent\textbf{Proof of Corollary {#1}.}
  \hspace*{0.em}}{\qed\\}
\newenvironment{proof-of-theorem}[1][{}]{\noindent\textbf{Proof of Theorem {#1}.}
  \hspace*{0.em}}{\qed\\}
\newenvironment{proof-attempt}{\noindent\textbf{Proof Attempt}
  \hspace*{0.em}}{\qed\bigskip\\}

\fi
\newtheorem{theorem}{Theorem}[section]
\newtheorem{lemma}{Lemma}[section]

\newtheorem{remark}{Remark}[section]

\usetikzlibrary{calc}

\allowdisplaybreaks
\usepackage[authoryear,round]{natbib}
\renewcommand*{\backref}[1]{\ifx#1\relax \else Page #1 \fi}
\renewcommand*{\backrefalt}[4]{%
  \ifcase #1 \footnotesize{(Not cited.)}%
  \or        \footnotesize{(Cited on page~#2.)}%
  \else      \footnotesize{(Cited on pages~#2.)}%
  \fi
}

\newcommand*{\colorboxed}{}
\def\colorboxed#1#{%
  \colorboxedAux{#1}%
}
\newcommand*{\colorboxedAux}[3]{%
  \begingroup
    \colorlet{cb@saved}{.}%
    \color#1{#2}%
    \boxed{%
      \color{cb@saved}%
      #3%
    }%
  \endgroup
}
\numberwithin{equation}{section}
\usepackage{xcolor} 
\usepackage{marginnote}

\newcommand{\todol}[2][]{{%
 \let\marginpar\marginnote
 \reversemarginpar
 \renewcommand{\baselinestretch}{0.8}%
 \todo[color=yellow]{#2}}}

\title{A Note on the Convergence of Muon}
\author{
{Jiaxiang Li} \thanks{Department of Electrical and Computer Engineering, University of Minnesota.  \texttt{li003755@umn.edu}}
\and
{Mingyi Hong} \thanks{Department of Electrical and Computer Engineering, University of Minnesota.  \texttt{mhong@umn.edu}}
}
\date{}

\begin{document}
\maketitle

\begin{abstract}
    In this note, we inspect the convergence of a new optimizer for pretraining LLMs, namely the Muon optimizer. Such an optimizer is closely related to a specialized steepest descent method where the update direction is the minimizer of the quadratic approximation of the objective function under spectral norm \citep{bernstein2024old}. We provide the convergence analysis on both versions of the optimizer and discuss its implications.
\end{abstract}

\section{Introduction}
Recently, a new optimizer named Muon\footnote{see \url{https://kellerjordan.github.io/posts/muon/}.} has drawn attentions due to its success in training large models. Consider the following stochastic optimization problem:
\begin{equation}
    \min_{X} f(X)=\E_{\xi} [F(x,\xi)]
\end{equation}
where the variable $X\in\RR^{m\times n}$ is a matrix. Without loss of generality we assume $m\geq n$. Muon optimizer writes: 
\begin{equation}\label{eq:muon}
\begin{aligned}
& G_t \leftarrow \nabla F\left(X_{t},\xi_t\right) \\
& B_t \leftarrow \mu B_{t-1} + G_t \\
& O_t \leftarrow \argmin_{O}\{\|O-B_t\|_F:O^\top O=I\text{ or }O O^\top=I\}\\
& X_{t+1} \leftarrow X_{t}-\eta_t O_t
\end{aligned}
\end{equation}
where the $O$ step can be equivalently written as $O_t=U V^\top$ where $B_t=U\Sigma V^\top$ is the singular value decomposition.



On the other hand, a closely related formula of \eqref{eq:muon} is the following spectral optimizer:
\begin{equation}\label{eq:muon_descent_form}
\begin{aligned}
& G_t \leftarrow \nabla F\left(X_{t},\xi_t\right) \\
& B_t \leftarrow \mu B_{t-1} + G_t \\
& \Delta_t \leftarrow \argmin_{\Delta}\{\tr(B_t^\top \Delta) + \frac{1}{2 \eta_t}\|\Delta\|_2^2\}\\
& X_{t+1} \leftarrow X_{t} + \Delta_t
\end{aligned}
\end{equation}

When $\|\cdot\|_2$ is the matrix 2-norm (spectral norm), i.e. the largest singular value, the solution of the third line of \eqref{eq:muon_descent_form} is $\Delta_t=-\eta_t \|B_t\|_{*} U V^\top$, where $B_t=U\Sigma V^\top$ is the singular value decomposition (see \citet[Proposition 5]{bernstein2024old}). In this case \eqref{eq:muon_descent_form} only differs from \eqref{eq:muon} by the norm $\|B_t\|_{*}$.

\section{Convergence of Muon}

Now we go back to the original Muon with a slight change: 
\begin{equation}\label{eq:muon2}
\begin{aligned}
& G_t \leftarrow \frac{1}{B}\sum_{i=1}^{B}\nabla F\left(X_{t},\xi_{t, i}\right), \\
& B_t \leftarrow \beta B_{t-1} + (1-\beta) G_t, \\
& O_t \leftarrow U_t V_t^\top,\text{ with SVD }B_t=U_t S_t V_t^\top, \\
& X_{t+1} \leftarrow X_{t}-\eta_t O_t,
\end{aligned}
\end{equation}
where in the second line $G_t$ is multiplied by $1-\beta$ as compared to \eqref{eq:muon}. \eqref{eq:muon2} is a heavy-ball method comparing to the Nesterov-type method \eqref{eq:muon}.

Note that since we assumed $X\in\RR^{m\times n}$ and $m\geq n$, the SVD satisfies $U_t\in\RR^{m\times r}$, $S_t\in\RR^{r\times r}$ and $V_t\in\RR^{n\times r}$, where $r$ is the rank of $B_t$, which we certainly need to assume to be positive. When $n=1$ ,i.e. the vector case, \label{eq:muon2} reduces to the normalized gradient descent with momentum. The convergence of normalized gradient descent with momentum is analyzed in \citet{cutkosky2020momentum}.

We now state the assumptions. Here the assumptions are with respect to the Frobenius norm. We need to assume the Lipschitzness of the gradient in Frobenius norm (a plain generalization of the Lipschitz smoothness in vector case), i.e.
\begin{equation}
    \|\nabla f(X) - \nabla f(Y)\|_{F}\leq L \|X- Y\|_{F}
\end{equation}
which implies \cite[Lemma 5.7]{beck2017first}
\begin{equation}\label{eq:lip_smooth}
    f(Y)\leq f(X) + \langle\nabla f(X), Y - X\rangle + \frac{L}{2}\|X-Y\|_{F}^2.
\end{equation}
The inner product is just the Frobenius inner product $\langle X, Y\rangle=\tr(X^\top Y)$. Another assumption is that $\nabla F(X,\xi)$ is an unbiased estimator of $\nabla f(X)$ with variance bounded by $\sigma^2$, i.e. $\E\|\nabla F(X,\xi) - \nabla f(X)\|_{F}^2\leq\sigma^2$. From now on we use $\|\cdot\|$ to denote the Frobenius norm throughout this section.

Now we move to convergence analysis. The main body of the analysis follows \cite{cutkosky2020momentum}, where we have a specific new descent lemma-type result for Frobenius norm.
\begin{lemma}\label{lemma:normalized_sgd_grad_bound}
    For update \eqref{eq:muon2}, we have that
    \begin{equation*}
        f(X_{t+1}) \leq f(X_{t}) -\frac{\eta_t}{4}\|\nabla f(X_{t})\| + \frac{5}{2} \eta_t\|{\nabla} f(X_{t}) - {B}_t\| + \frac{\eta_t^2 n L}{2}.
    \end{equation*}
    If we take $\eta_t=\eta$ a constant, we have
    \begin{equation*}
        \sum_{t=1}^{T} \|\nabla f(X_t)\| \leq \frac{4(f(X_1)-f^*)}{\eta} + 10 \sum_{t=1}^{T}\|\nabla f(X_t) - B_t\| + 2\eta n L T.
    \end{equation*}
\end{lemma}
\begin{proof}
    By Lipschitz smoothness we have
    \begin{equation}\label{eq:temp1.5}
    \begin{aligned}
        f(X_{t+1}) \leq & f(X_{t}) + \langle\nabla f(X_t), X_{t+1} - X_t\rangle + \frac{L}{2}\|X_{t+1} - X_t\|^2 \\
        = & f(X_{t}) -\eta_t \langle\nabla f(X_t), O_t\rangle + \frac{\eta_t^2 n L}{2}
    \end{aligned}
    \end{equation}
    since $\|X_{t+1} - X_t\|=\eta_t\|O_t\| = \eta_t \sqrt{n}$.

    Now we analyze the term $-\eta_t\langle\nabla f(X_t), O_t\rangle$. Suppose $B_t=U_t S_t V_t^\top$ being the SVD decomposition of $B_t$ so that $O_t = U_t V_t^\top$. Then
    \begin{equation}\label{eq:temp2}
    \begin{split}
        -\langle\nabla f(X_t), O_t\rangle = & -\tr(\nabla f(X_t)^\top U_t V_t^\top) \\
        = & -\tr(\nabla f(X_t)^\top U_t S_t V_t^\top V_t S_t^{-1} V_t^\top) \\
        = & -\tr(\nabla f(X_t)^\top B_t V_t S_t^{-1} V_t^\top) \\
        = & -\langle \nabla f(X_t) V_t S_t^{-1/2}, B_t V_t S_t^{-1/2}\rangle
    \end{split}
    \end{equation}

    For simplicity denote $\nabla_t= \nabla f(X_t)$, $\Tilde{\nabla}_t=\nabla f(X_t) V_t S_t^{-1/2}$ and $\Tilde{B}_t=B_t V_t S_t^{-1/2}$. We have
    \begin{equation}\label{eq:temp3}
    \begin{split}
        &\|\nabla_t\| = \|\nabla_t V_t S_t^{-1/2} S_t^{1/2} V_t^\top\|\leq \|\Tilde{\nabla}_t\| \|S_t^{1/2}\|_2 \Rightarrow \|\Tilde{\nabla}_t\|\geq \frac{\|\nabla_t\|}{\|S_t^{1/2}\|_2} \\
        &\|\Tilde{\nabla}_t - \Tilde{B}_t\| = \|({\nabla}_t - {B}_t) V_t S_t^{-1/2}\| \leq \|({\nabla}_t - {B}_t)\|\|S_t^{-1/2}\|_2
    \end{split}
    \end{equation}

    Now back to \eqref{eq:temp2}, we have
    \begin{equation}\label{eq:temp4}
    \begin{split}
        -\langle\nabla f(X_t), O_t\rangle = & -\langle \nabla f(X_t) V_t S_t^{-1/2}, B_t V_t S_t^{-1/2}\rangle = -\langle \Tilde{\nabla}_t, \Tilde{B}_t\rangle \\
        = & \frac{1}{2}\left( \|\Tilde{\nabla}_t - \Tilde{B}_t\|^2 - \|\Tilde{\nabla}_t\|^2 - \|\Tilde{B}_t\|^2\right) \\
        \eqref{eq:temp3}\leq & \frac{1}{2}\left( \|({\nabla}_t - {B}_t)\|^2\|S_t^{-1/2}\|_2^2 - \frac{\|\nabla_t\|^2}{\|S_t^{1/2}\|_2^2} \right) \\
        \leq & \frac{\|({\nabla}_t - {B}_t)\|^2 - \|\nabla_t\|^2}{2\|S_t^{1/2}\|_2^2} = \frac{\|({\nabla}_t - {B}_t)\|^2 - \|\nabla_t\|^2}{2\|S_t\|_2} \\
        = & \frac{\|{\nabla}_t - {B}_t\|^2 - \|\nabla_t\|^2}{2\|{B}_t\|_2}
    \end{split}
    \end{equation}
    where the last inequality is due to $\|S^{-1}\|_2\leq 1/\|S\|_2$. Now if $2 \|{\nabla}_t - {B}_t\| \leq \|\nabla_t\|$, we have (note that $\|{B}_t\|_2\leq \|{B}_t\|$)
    \begin{equation*}
    \begin{split}
        -\langle\nabla f(X_t), O_t\rangle \leq & \frac{\|{\nabla}_t - {B}_t\|^2 - \|\nabla_t\|^2}{2\|{B}_t\|_2} \\
        \leq & -\frac{3}{8}\frac{\|\nabla_t\|^2}{\|{B}_t\|_2} \leq -\frac{3}{8}\frac{\|\nabla_t\|^2}{\|{B}_t\|} \\
        = & -\frac{3}{8}\frac{\|\nabla_t\|^2}{\|{B}_t - {\nabla}_t + {\nabla}_t\|} \\
        \leq & -\frac{3}{8}\frac{\|\nabla_t\|^2}{\|{B}_t - {\nabla}_t\| + \|{\nabla}_t\|} \\
        \leq & -\frac{3}{8}\frac{\|\nabla_t\|^2}{\frac{1}{2}\|{\nabla}_t\| + \|{\nabla}_t\|} = -\frac{1}{4} \|\nabla_t\|
    \end{split}
    \end{equation*}

    Otherwise if $2 \|{\nabla}_t - {B}_t\| > \|\nabla_t\|$, we have 
    \begin{equation*}
    \begin{split}
        -\langle\nabla f(X_t), O_t\rangle \leq & \|\nabla_t\| \|O_t\|_2 = \|\nabla_t\| \\
        = & -\frac{1}{4}\|\nabla_t\| + \frac{5}{4}\|\nabla_t\| \\
        \leq & -\frac{1}{4}\|\nabla_t\| + \frac{5}{2} \|{\nabla}_t - {B}_t\|
    \end{split}
    \end{equation*}

    So in either cases we have
    \begin{equation}
        -\langle\nabla f(X_t), O_t\rangle \leq -\frac{1}{4}\|\nabla_t\| + \frac{5}{2} \|{\nabla}_t - {B}_t\|
    \end{equation}
    Plugging this back to \eqref{eq:temp1.5} we get
    \begin{equation}\label{eq:temp5}
        f(X_{t+1}) \leq f(X_{t}) -\frac{\eta_t}{4}\|\nabla f(X_{t})\| + \frac{5}{2} \eta_t\|{\nabla} f(X_{t}) - {B}_t\| + \frac{\eta_t^2 n L}{2}
    \end{equation}
    and the second equation in the lemma statement is obtained by summing up \eqref{eq:temp5} from $t=1,...,T$.
\end{proof}

\begin{theorem}
    Let $R=f(X_1)-f^*$. If we take $\beta=1-\alpha$ with $\alpha=\min(\frac{\sqrt{RL}}{\sigma \sqrt{T}}, 1)$, also $\eta_t=\eta=\sqrt{\frac{4 R}{(10/\alpha + 2n) T L}}$ and $B=1$ (batch free convergence), then for update \eqref{eq:muon2} we have
    \begin{equation*}
    \frac{1}{T} \sum_{t=1}^T \mathbb{E}\left[\left\|\nabla f\left(X_t\right)\right\|\right] \leq \mathcal{O}\left(\frac{\sqrt{n R L}}{\sqrt{T}}+\frac{\sigma^2}{\sqrt{R L T}}+\frac{\sqrt{\sigma}(R L)^{1 / 4}}{T^{1 / 4}}\right).
    \end{equation*}

    If we take $\beta$ as an arbitrary constant, then we will need to take $B=T$, so that
    \begin{equation*}
        \frac{1}{T}\sum_{t=1}^{T} \|\nabla f(X_t)\| \leq \mathcal{O}\left(\frac{\sqrt{n R L}}{\sqrt{T}}+\frac{\sigma}{T^{3 / 2}}+\frac{\sigma}{\sqrt{T}}\right).
    \end{equation*}
\end{theorem}
\begin{proof}
    The proof follows \cite{cutkosky2020momentum}. Denote $\hat{\delta}_t=B_t-\nabla f(X_t)$, ${\delta}_t=G_t-\nabla f(X_t)$ and $S(X,Y)=\nabla f(X) - \nabla f(Y)$. Note that we have
    \begin{equation}
    \begin{aligned}
        &\E [\delta_t]=0,\ \E\left[\|\delta_t\|^2\right]\leq \frac{\sigma^2}{m},\\
        &\E[\langle\delta_i, \delta_j\rangle]=0,\ \forall i\not=j \\
        & \|S(X,Y)\|\leq L\|X - Y\|
    \end{aligned}
    \end{equation}

    Now following the update in \eqref{eq:muon2}, we get
    \begin{equation*}
    \begin{split}
        \hat{\delta}_{t+1} & = \beta \hat{\delta}_{t} + (1-\beta) \delta_t + S(X_{t}, X_{t+1}) \\
        & = \beta^{t} \hat{\delta}_{1} + (1 - \beta) \sum_{\tau=0}^{t-1}\beta^{\tau}\delta_{t-\tau} + \sum_{\tau=0}^{t-1}\beta^{\tau}S(X_{t-\tau}, X_{t+1-\tau}),
    \end{split}
    \end{equation*}
    therefore 
    \begin{equation*}
        \|\hat{\delta}_{t+1}\|\leq \beta^{t} \|\hat{\delta}_{1}\| + (1 - \beta) \|\sum_{\tau=0}^{t-1}\beta^{\tau}\delta_{t-\tau}\| + \eta L \sum_{\tau=0}^{t-1}\beta^{\tau}.
    \end{equation*}
    Taking expectation we get (using the fact that $\hat{\delta}_1 = \delta_1$) 
    \begin{equation*}
    \begin{split}
        \E\|\hat{\delta}_{t+1}\| \leq & \beta^{t} \frac{\sigma}{\sqrt{B}} + (1 - \beta) \sqrt{\sum_{\tau=0}^{t-1}\beta^{2\tau}\frac{\sigma^2}{B}} + \eta L \sum_{\tau=0}^{t-1}\beta^{\tau}\\
        \leq & \frac{\sigma}{\sqrt{B}}\beta^{t} + \frac{\sigma}{\sqrt{B}} \frac{1 - \beta}{\sqrt{1-\beta^2}} + \eta L \frac{1}{1-\beta} \\
        \leq & \frac{\sigma}{\sqrt{B}}\beta^{t} + \frac{\sigma}{\sqrt{B}} \sqrt{1-\beta} + \eta L \frac{1}{1-\beta}.
    \end{split}
    \end{equation*}

    In conclusion we get
    \begin{equation*}
        \sum_{t=1}^{T}\E\|\hat{\delta}_{t+1}\|\leq \frac{\sigma}{(1-\beta)\sqrt{B}} + T\sqrt{1-\beta}\frac{\sigma}{\sqrt{B}} + \frac{T\eta L}{1-\beta}.
    \end{equation*}
    Now utilize Lemma \ref{lemma:normalized_sgd_grad_bound}, we get
    \begin{equation*}
    \begin{split}
        \sum_{t=1}^{T} \|\nabla f(X_t)\| \leq & \frac{4(f(X_1)-f^*)}{\eta} + 10 \sum_{t=1}^{T}\|\nabla f(X_t) - B_t\| + 2\eta n L T \\
        \leq & \frac{4(f(X_1)-f^*)}{\eta} + 10 \frac{\sigma}{(1-\beta)\sqrt{B}} + 10T\sqrt{1-\beta}\frac{\sigma}{\sqrt{B}} + 10\frac{T\eta L}{1-\beta} + 2\eta n L T \\
        \leq & \frac{4R}{\eta} + 10 \frac{\sigma}{(1-\beta)\sqrt{B}} + 10T\sqrt{1-\beta}\frac{\sigma}{\sqrt{m}} + 10\frac{T\eta L}{1-\beta} + 2\eta n L T.
    \end{split}
    \end{equation*}

    Now we just need to take $\eta=\sqrt{\frac{4 R}{(10/(1-\beta) + 2n) T L}}$ so that
    \begin{equation*}
        \sum_{t=1}^{T} \|\nabla f(X_t)\| \leq 2 \sqrt{(10/(1-\beta) + 2n) R T L} + 10 \frac{\sigma}{(1-\beta)\sqrt{B}} + 10T\sqrt{1-\beta}\frac{\sigma}{\sqrt{B}}.
    \end{equation*}

    Now we have two types of parameter choice. If we take $B=1$ (batch size free), we need to take $1-\beta=\min(1, \frac{\sqrt{RL}}{\sigma\sqrt{T}})$ so that we have
    \begin{equation*}
        \sum_{t=1}^{T} \|\nabla f(X_t)\| \leq 2\sqrt{2n RTL} + 10\sigma^2\sqrt{\frac{T}{RL}} + 2\sqrt{10\sigma}(RL)^{1/4}T^{3/4}+10\sqrt{\sigma}(RL)^{1/4}T^{3/4},
    \end{equation*}
    thus
    \begin{equation*}
        \frac{1}{T}\sum_{t=1}^{T} \|\nabla f(X_t)\| \leq \mathcal{O}\left(\frac{\sqrt{n R L}}{\sqrt{T}}+\frac{\sigma^2}{\sqrt{R L T}}+\frac{\sqrt{\sigma}(R L)^{1 / 4}}{T^{1 / 4}}\right).
    \end{equation*}

    If we take $\beta$ as an arbitrary constant in $(0, 1)$, then we will need to take $B=T$, so that 
    \begin{equation*}
        \frac{1}{T}\sum_{t=1}^{T} \|\nabla f(X_t)\| \leq \mathcal{O}\left(\frac{\sqrt{n R L}}{\sqrt{T}}+\frac{\sigma}{T^{3 / 2}}+\frac{\sigma}{\sqrt{T}}\right).
    \end{equation*}

    If we take $\beta$ as an arbitrary constant in $(0, 1)$ and take $B=T^{\alpha}$ where $\alpha\in(0, 1)$, we have
    \begin{equation*}
        \frac{1}{T}\sum_{t=1}^{T} \|\nabla f(X_t)\| \leq \mathcal{O}\left(\frac{\sqrt{n R L}}{\sqrt{T}}+\frac{\sigma}{T^{\alpha / 2+1}}+\frac{\sigma}{T^{\alpha/2}}\right).
    \end{equation*}
    The proof is completed.
\end{proof}

\section{Convergence of the generic scheme in \cite{bernstein2024old}}
Now we analyze the convergence of \eqref{eq:muon_descent_form}, 
where we actually analyze the following heavy-ball and mini-batch scheme:
\begin{equation}\label{eq:muon_descent_form2}
\begin{aligned}
& G_{t} \leftarrow \frac{1}{B}\sum_{i=1}^{B}\nabla F\left(X_{t},\xi_{t, i}\right) \\
& B_t \leftarrow \beta B_{t-1} + (1 - \beta) G_{t} \\
& \Delta_t \leftarrow \argmin_{\Delta}\{\tr(B_{t}^\top \Delta) + \frac{1}{2 \eta}\|\Delta\|^2\}\\
& X_{t+1} \leftarrow X_{t} + \Delta_t
\end{aligned}
\end{equation}
where $\|\cdot\|$ denotes the spectral norm, and $\|\cdot\|_{*}$ is the dual norm, i.e. nuclear norm. Such notations carry on in this section.

We need to assume the Lipschitzness of the gradient, i.e.
\begin{equation}
    \|\nabla f(X) - \nabla f(Y)\|_{*}\leq L \|X- Y\|
\end{equation}
which implies \cite[Lemma 5.7]{beck2017first}
\begin{equation}\label{eq:lip_smooth_dual_norm}
    f(Y)\leq f(X) + \langle\nabla f(X), Y - X\rangle + \frac{L}{2}\|X-Y\|^2.
\end{equation}
Here the inner product is just the Frobenius inner product $\langle X, Y\rangle=\tr(X^\top Y)$. Another assumption is that $\nabla F(X,\xi)$ is an unbiased estimator of $\nabla f(X)$ with variance bounded by $\sigma^2$, i.e. 
\begin{equation}
    \E\|\nabla F(X,\xi) - \nabla f(X)\|_{*}^2\leq \sigma^2.
\end{equation}

We have the following result:
\begin{theorem}
    Let $R=f(X_1)-f^*$. If we take $\eta \leq\frac{1}{8L}\sqrt{\frac{1-2\beta}{2\beta}}$ and $\beta$ as any constant in $(0,1)$, then for update \eqref{eq:muon_descent_form2} we have
    \begin{equation*}
    \begin{aligned}
        \frac{1}{T}\sum_{t=1}^{T}\E\|\nabla f(X_t)\|_{*}^2
        \leq \mathcal{O}\left(\frac{R}{T} + \frac{2 n \sigma^2}{B}\right).
    \end{aligned}
    \end{equation*}
\end{theorem}

\begin{proof}
We first inspect the third line of \eqref{eq:muon_descent_form2}. We know that \cite[Proposition 1]{bernstein2024old}
$$
\argmin_{\Delta}\{\tr(B^\top \Delta) + \frac{1}{2 \eta}\|\Delta\|_2^2\} = -\eta \|B\|_{*} \argmax_{\|A\|=1} \langle A, B\rangle
$$
where $A^*:=\argmax_{\|A\|=1} \langle A, B\rangle = U V^\top$ is exactly the matrix such that $\langle A^*, B\rangle=\|B\|_{*}$, where $B=U\Sigma V^\top$ is the SVD. Therefore we have the following useful results:
\begin{equation}\label{eq:useful_eqs}
\begin{aligned}
    & \langle\Delta_t, B_t\rangle = -\eta \|B_t\|_{*}^2 \\
    &\|\Delta_t\| = \eta \|B_t\|_{*}
\end{aligned}
\end{equation}

Now we analyze $B_t - \nabla f(X_t)$, by convexity
\begin{equation}\label{eq:bt_grad_error}
\begin{aligned}
    & \|B_t - \nabla f(X_t)\|_{*}^2 \\
    = & \|(1-\beta) (G_t - \nabla f(X_t)) + \beta (B_{t-1} - \nabla f(X_t))\|_{*}^2 \\
    \leq & (1-\beta) \|G_t - \nabla f(X_t)\|_{*}^2 + \beta\|B_{t-1} - \nabla f(X_t)\|_{*}^2 \\
    \leq & (1-\beta) \|G_t - \nabla f(X_t)\|_{*}^2 + 2\beta\|B_{t-1} - \nabla f(X_{t-1})\|_{*}^2 + 2\beta\|\nabla f(X_{t}) - \nabla f(X_{t-1})\|_{*}^2.
\end{aligned}
\end{equation}
Now if we only take expectation w.r.t $\xi_t$ (random variables of iteration $t$) we get 
\begin{equation}\label{eq:variance_spectral_norm}
\begin{aligned}
    \E\|G_t - \nabla f(X_t)\|_{*}^2 \leq & n\E\|G_t - \nabla f(X_t)\|_{F}^2 \\
    = & n \E\| \frac{1}{B}\sum_{i=1}^{B}\left(\nabla F(X_{t},\xi_{t, i} - \nabla f(X_t)\right)\|_{F}^2 \\
    \stackrel{(i)}= & \frac{n}{B}\E\| \left(\nabla F(X_{t},\xi_{t, 1} - \nabla f(X_t)\right)\|_{F}^2 \\
    \leq & \frac{n}{B}\E\| \left(\nabla F(X_{t},\xi_{t, 1} - \nabla f(X_t)\right)\|_{*}^2 \leq \frac{n}{B}\sigma^2,
\end{aligned}
\end{equation}
where $(i)$ is due to the independency of the samples $\xi_{t,i}$.

Therefore
\begin{equation*}
\begin{aligned}
    & \E\|B_t - \nabla f(X_t)\|_{*}^2 \\
    \leq & (1-\beta) \frac{n\sigma^2}{B} + (1+\gamma)\beta\|B_{t-1} - \nabla f(X_{t-1})\|_{*}^2 + (1+\frac{1}{\gamma})\beta L^2\|X_t - X_{t-1}\|^2 \\
    = & (1-\beta) \frac{n\sigma^2}{B} + (1+\gamma)\beta\|B_{t-1} - \nabla f(X_{t-1})\|_{*}^2 + (1+\frac{1}{\gamma})\beta L^2\|\Delta_t\|^2 \\
    \leq & (1+\gamma)\beta\|B_{t-1} - \nabla f(X_{t-1})\|_{*}^2 + (1-\beta) \frac{n\sigma^2}{B} + (1+\frac{1}{\gamma})\beta L^2 \eta^2 \|B_t\|_{*}^2,
\end{aligned}
\end{equation*}
where $\gamma>0$ is a constant to be determined.

Therefore we have
\begin{equation*}
\begin{aligned}
    (1-(1+\gamma)\beta) & \sum_{t=1}^{T}\E\|B_t - \nabla f(X_t)\|_{*}^2 \\
    \leq &  (1-\beta) \frac{n\sigma^2}{B} T + (1+1/\gamma)\beta L^2 \eta^2 \sum_{t=1}^{T}\E\|B_t\|_{*}^2.
\end{aligned}
\end{equation*}

Now since 
\begin{equation*}
\sum_{t=1}^{T}\E\|B_t\|_{*}^2 \leq 2\sum_{t=1}^{T}\E\|\nabla f(X_t)\|_{*}^2 + 2\sum_{t=1}^{T}\E\|B_t - \nabla f(X_t)\|_{*}^2
\end{equation*}
we have
\begin{equation}\label{eq:temp1}
\begin{aligned}
    &(1 - (1+\gamma)\beta - 2(1+1/\gamma) \beta L^2 \eta^2) \sum_{t=1}^{T}\E\|B_t - \nabla f(X_t)\|_{*}^2 \\
    \leq & (1-\beta) \frac{n\sigma^2}{B} T + 2(1+1/\gamma)\beta L^2 \eta^2 \sum_{t=1}^{T}\E\|\nabla f(X_t)\|_{*}^2.
\end{aligned}
\end{equation}

Now by \eqref{eq:lip_smooth_dual_norm} we have
\begin{equation*}
\begin{aligned}
    f(X_{t+1}) \leq & f(X_t) + \langle\nabla f(X_t), \Delta_t\rangle + \frac{L}{2}\|\Delta_t\|^2 \\
    \text{\eqref{eq:useful_eqs}}= & f(X_t) - \eta \|B_t\|_{*}^2 + \frac{\eta^2 L}{2}\|B_t\|_{*}^2 + \langle\nabla f(X_t) - B_t, \Delta_t\rangle \\
    \leq & f(X_t) - (\eta - \frac{\eta^2 L}{2})\|B_t\|_{*}^2 + \frac{c}{2}\|\nabla f(X_t) - B_t\|_{*}^2 + \frac{1}{2c} \eta^2 \|B_t\|_{*}^2
\end{aligned}
\end{equation*}
so we get
\begin{equation*}
\begin{aligned}
    (\eta - \frac{\eta^2 L}{2} - \frac{\eta^2}{2c})\|B_t\|_{*}^2 \leq & f(X_t) - f(X_{t+1}) + \frac{c}{2}\|\nabla f(X_t) - B_t\|_{*}^2.
\end{aligned}
\end{equation*}
Therefore the following holds:
\begin{equation*}
\begin{aligned}
    &(\eta - \frac{\eta^2 L}{2} - \frac{\eta^2}{2c})\|\nabla f(X_t)\|_{*}^2 \\
    \leq & 2(\eta - \frac{\eta^2 L}{2} - \frac{\eta^2}{2c})\|B_t\|_{*}^2 + 2 (\eta - \frac{\eta^2 L}{2} - \frac{\eta^2}{2c})\|\nabla f(X_t) - B_t\|_{*}^2\\
    \leq & 2(f(X_t) - f(X_{t+1})) + \left(c+(2\eta - \eta^2 L - \frac{\eta^2}{c})\right)\|\nabla f(X_t) - B_t\|_{*}^2.
\end{aligned}
\end{equation*}

Taking $c = \eta$ we get
\begin{equation}\label{eq:grad_norm_sq_bound}
\begin{aligned}
    &(\frac{\eta}{2}- \frac{\eta^2 L}{2})\|\nabla f(X_t)\|_{*}^2 \leq 2(f(X_t) - f(X_{t+1})) + \left(2\eta - \eta^2 L\right)\|\nabla f(X_t) - B_t\|_{*}^2.
\end{aligned}
\end{equation}

Now sum up above inequality we get
\begin{equation*}
\begin{aligned}
    &(\frac{\eta}{2}- \frac{\eta^2 L}{2})\sum_{t=1}^{T}\E\|\nabla f(X_t)\|_{*}^2 \\
    \leq & 2(f(X_0) - f^*) + \left(2\eta - \eta^2 L\right)\sum_{t=1}^{T}\E\|\nabla f(X_t) - B_t\|_{*}^2.
\end{aligned}
\end{equation*}
Now apply \eqref{eq:temp1} to above equation we get 
\begin{equation*}
\begin{aligned}
    &(\frac{\eta}{2}- \frac{\eta^2 L}{2})\sum_{t=1}^{T}\E\|\nabla f(X_t)\|_{*}^2 \\
    \leq & 2(f(X_0) - f^*) + \frac{2\eta - \eta^2 L}{(1 - (1+\gamma)\beta - 2(1+1/\gamma) \beta L^2 \eta^2)}\left( (1-\beta) \frac{n\sigma^2}{m} T + 2(1+1/\gamma)\beta L^2 \eta^2  \sum_{t=1}^{T}\E\|\nabla f(X_t)\|_{*}^2\right).
\end{aligned}
\end{equation*}
That is, the following holds:
\begin{equation*}
\begin{aligned}
    &\left(\frac{\eta}{2}- \frac{\eta^2 L}{2} - \frac{2(2\eta - \eta^2 L)(1+1/\gamma)\beta L^2 \eta^2}{(1 - (1+\gamma)\beta - 2(1+1/\gamma) \beta L^2 \eta^2)}\right)\sum_{t=1}^{T}\E\|\nabla f(X_t)\|_{*}^2 \\
    \leq & 2(f(X_0) - f^*) + \frac{(1-\beta)(2\eta - \eta^2 L)}{(1 - (1+\gamma)\beta - 2(1+1/\gamma) \beta L^2 \eta^2)} \frac{n\sigma^2}{m} T.
\end{aligned}
\end{equation*}

Taking $\gamma\leq 1/(2\beta)-1$ so that $1 - (1+\gamma)\beta\geq 1/2$, also taking $\eta \leq\frac{1}{8L}\sqrt{\frac{1-2\beta}{2\beta}}$ and $\beta$ as any constant in $(0,1)$, we get
    \begin{equation*}
    \begin{aligned}
        \frac{1}{T}\sum_{t=1}^{T}\E\|\nabla f(X_t)\|_{*}^2
        \leq \mathcal{O}\left(\frac{R}{T} + \frac{2 n \sigma^2}{B}\right).
    \end{aligned}
    \end{equation*}
This completes the proof. 
\end{proof}

\begin{remark}
    Note that we are not able to achieve a batch free convergence for \eqref{eq:muon_descent_form2}. The difficulty stems from \eqref{eq:bt_grad_error}, where we used convexity of the quadratic and norm functions. A better bound is available in \citet[Theorem 3.5]{ghadimi2020single} yet it only works for vector norm to our knowledge.
\end{remark}

\begin{remark}
    If we replace the norm in this section by an arbitrary norm, \eqref{eq:variance_spectral_norm} no longer holds. In stead we can only say $\E\|G_t - \nabla f(X_t)\|_{*}^2 \leq\sigma^2$ and the final bound becomes
    \begin{equation*}
    \begin{aligned}
        \frac{1}{T}\sum_{t=1}^{T}\E\|\nabla f(X_t)\|_{*}^2
        \leq \mathcal{O}\left(\frac{R}{T} + \sigma^2\right).
    \end{aligned}
    \end{equation*}
    In this case there is an unavoidable constant variance term even with mini-batch updates.
\end{remark}

\section*{Acknowledgment}
JL and MH thank Jeremy Bernstein, also Zhiqi Bu for helpful discussions.

\clearpage
\bibliographystyle{abbrvnat} 
\bibliography{reference}

\end{document}